\documentclass{amsart}
\usepackage{sty}

\usepackage[margin=1.5in]{geometry}
\usepackage{graphicx}
\usepackage[pdftex]{hyperref}\hypersetup{colorlinks,citecolor=blue,linktocpage}
\usepackage{xcolor}

\usepackage{tikz}
\usetikzlibrary{shapes}
\usetikzlibrary{decorations.markings}
\tikzset{every picture/.style={>=latex}}
\tikzstyle{vert}=[scale=.8,
      style={circle,fill=black,minimum size=5pt,inner sep=0pt}]
\tikzstyle{label}=[scale=.9]

\author{Valery Alexeev and Adrian Brunyate}
\address{Department of Mathematics, University of Georgia}
\email{valery@math.uga.edu, brunyate@math.uga.edu}
\title[Extending the Torelli map]%
{Extending the Torelli map to toroidal\\ compactifications of Siegel space}
\date{February 23, 2011; corrected: April 27 and May 25, 2011}

\begin{document}
\maketitle

\begin{abstract}
  It has been known since the 1970s that the Torelli map $\M_g \to
  \A_g$, associating to a smooth curve its jacobian, extends to a
  regular map from the Deligne-Mumford compactification $\oM_g$ to the
  2nd Voronoi compactification $\oA_g\vor$.  We prove that the
  extended Torelli map to the perfect cone (1st Voronoi)
  compactification $\oA_g\perf$ is also regular, and moreover
  $\oA_g\vor$ and $\oA_g\perf$ share a common Zariski open
  neighborhood of the image of $\oM_g$. We also show that the map to
  the Igusa monoidal transform (central cone compactification) is
  \emph{not} regular for $g\ge9$; this disproves a 1973 conjecture of
  Namikawa.
\end{abstract}

\tableofcontents

\section*{Introduction}

The Torelli map $\M_g \to \A_g$ associates to a smooth curve $C$ its
jacobian $JC$, a principally polarized abelian variety.  Does it
extend to a regular map $\oM_g\to \oA_g$ between the compactified
moduli spaces?

For the moduli space of curves $\M_g$, a somewhat canonical choice of
a compactification is provided by the Deligne-Mumford
compactification $\oM_g$, which we fix for the remainder of the paper.

We note in passing that recently other compactifications
$\oM_g(\alpha)$ were considered by many authors. These are log
canonical models of $\oM_g$ with respect to $K_{\oM_g}+\alpha \delta$,
where $\delta$ is the boundary. They also have modular
interpretation. For example, for $9/11\ge\alpha> 7/10$, assuming
$g\ge3$, $\oM_g(\alpha)$ is the moduli spaces of curves with nodes and
cusps and without elliptic tails. However, the extended map
$\oM_g(\alpha)\to \oA_g$ has no chance of being regular (unless
$\oM_g(\alpha)= \oM_g$) because curves of compact type with elliptic
tails map to the interior $\A_g$, which has to be suitably modified as
well.

For the moduli space of principally polarized abelian varieties
$\A_g$, by \cite{AMRT} there are infinitely many choices of toroidal
compactifications $\oA_g^{\tau}$, each determined by a fan $\tau$
supported on the space of positive semidefinite quadratic forms in
$g$ variables, periodic w.r.t. $\GL(g,\bZ)$, with only finitely many
orbits. There are three standard explicit choices for $\tau$, and they
all have interesting geometric meanings:
\begin{enumerate}
\item 1st Voronoi fan = perfect cones $\tau\perf$,
\item 2nd Voronoi fan = Delaunay-Voronoi fan = L-type domains $\tau\vor$,
\item central cones $\tau\cent$.
\end{enumerate}
The first two of these were defined by G. Voronoi in a series of
papers \cite{Voronoi08all} on reduction theory
of quadratic forms, published posthumously in 1908-9. 

The 2nd Voronoi compactification appears in \cite{Alexeev_CMAV} as the
normalization of the main irreducible component of the moduli space
$\overline{\operatorname{AP}}_g$ of stable semiabelic pairs
$(X,\Theta)$ which provides a moduli compactification of $\A_g$. On
the other hand, by \cite{ShepherdBarron_Perfect} the perfect cone
compactification $\oA_g\perf$ is the canonical model of any smooth
compactification of $\A_g$, if $g\ge 12$ (and also for all $g$, if
considered as stacks and relatively over Satake-Baily-Borel
compactification $\oA_g^*$).

The central cones fan was introduced by Igusa \cite{Igusa_BlowUp}; the
corresponding toroidal compactification $\oA_g\cent$ is the
normalization of the blowup of the Satake-Baily-Borel compactification
$\oA_g^*$ along the boundary (the ``Igusa blowup'').

The basic question we consider is this: for which choices of a fan
$\tau$ does the Torelli map $\M_g\to\A_g$ extend to a regular map map
$\oM_g\to\oA_g\utau$? For the 2nd Voronoi fan, a positive answer was
given by Mumford and Namikawa \cite[\S18]{Namikawa_NewCompBoth}. This
prompted an extensive study of the 2nd Voronoi compactification by
Namikawa \cite{Namikawa_NewCompBoth}, continued in the construction of
the moduli of stable semiabelic pairs $\overline{\operatorname{AP}}_g$
in \cite{Alexeev_CMAV}.  The work \cite{Alexeev_CompJacobians} gives a
modular interpretation for the extended Torelli map $\oM_g\to
\overline{\operatorname{AP}}_g$.

Historically, the extension question for the Igusa blowup was the
first one to be considered, in a pioneering 1973 paper
\cite{Namikawa_ExtendedTorelli} of Namikawa. There, it is shown that
$\oM_g\to\oA_g\cent$ is regular for low $g$ (the bound $g\le 6$ is
stated without proof), regular on the locus of curves with
a planar dual graph, and conjectured that the map is regular for all
$g$.

The question for the perfect cone compactifications was not previously
considered, to our knowledge. 

In this paper, we prove that the extended map is regular for the
perfect cone compactification for all $g$. Much more than that, we
prove that the perfect and the 2nd Voronoi compactifications share
a common open neighborhood of the image of $\oM_g$. Note
that in general there is a birational map $\oA_g\vor \dashrightarrow
\oA_g\perf$ which does not create new divisors. It is an isomorphism
iff $g\le3$, and regular for $g\le5$. According to
\cite{ErdahlRybnikov_VoronoiDickson, ErdahlRybnikov_CompareVoronois},
this map is not regular for $g\ge6$. Thus, for higher $g$ the two
compactifications are truly different, but we prove that they are equal
near the closure of the Schottky locus.

For the central cone compactification, we prove that the extended map
is regular for $g\le6$ and is \emph{not} regular for $g\ge9$. 
Continuing the methods of the present paper, \cite{Vigre_Genus7}
also settled the cases $g=7,8$ positively, by a lengthy computation. 

The structure of the paper is as follows. In
Section~\ref{sec:toroidal-comp} we recall the combinatorial data for a
toroidal compactification of $\A_g$, and define the fans $\tau\perf$,
$\tau\cent$, $\tau\vor$. 

In Section~\ref{sec:graphs} we fix the notations for graphs and define
\emph{edge-minimizing metrics}, which we abbreviate to \emph{emm}, on
the first cohomology group $H^1(G, \bZ)$ of a graph. 

In Section~\ref{sec:extension-general} we give a general
criterion for the regularity of the extended Torelli map
$\oM_g\to\oA_g\utau$, and illustrate it in the case of the 2nd Voronoi
compactification, as proved by Mumford and Namikawa. Then we reduce the cases of
perfect cones, resp. central cones, to the existence of an
$\bR$-valued, resp. a $\bZ$-valued emm for any graph of genus $\le
g$. 
We also prove that the existence of a strong $\bR$-emm implies that
$\oA_g\perf$ and $\oA_g\vor$ share a common open neighborhood of the
image of $\oM_g$.

In Section~\ref{sec:Z-emms}, we prove that a $\bZ$-emm exists for any
graph of genus $g\le 6$, and does not exist for some explicit graphs
of genus $9$, thus settling negatively the extension question for the
central cone compactification and $g\ge9$.

In Section~\ref{sec:R-emms}, we prove that a strong $\bR$-emm exists
for any graph, thus proving the regularity of $\oM_g\to\oA_g\perf$ and
the statement about a common neighborhood.

In the concluding Section~\ref{sec:misc} we discuss some possible
extensions of out results.

\begin{acknowledgments}
  The first author was partially supported by NSF under
  DMS-0901309. The subject of the paper was the topic of a VIGRE
  research group at the University of Georgia in the Fall of 2010, led
  by the first author. We would like to thank the participants of the
  group for many discussions, and to acknowledge NSF's VIGRE support
  under DMS-0738586. We also thank the referee for very helpful
  suggestions for improvements.
\end{acknowledgments}

\section{Toroidal compactifications of $\A_g$}
\label{sec:toroidal-comp}

Here, $\A_g$ stands for the moduli space of principally polarized
abelian varieties. The theory of its toroidal compactifications over $\bC$
was developed by Mumford and his coworkers in \cite{AMRT};
\cite{FaltingsChai} contains an extension to the arithmetic case,
over~$\bZ$.  It is parallel to the theory of ordinary toric varieties.

As in toric geometry, there are two dual lattices, $M$ (for monomials)
and $N$ (for 1-parameter subgroups in the torus).  The real vector
space $N_{\bR}$ is the ambient space for a fan $\tau$, and $M_{\bR}$
is the ambient space for polyhedra.  For compactifications of $\A_g$,
one fixes a free abelian group $\Lambda\simeq\bZ^g$. Then $M=\Sym^2
\Lambda$, and $N = \Gamma^2\Lambda^*$ is the dual abelian group, the
second divided power of $\Lambda^*$.

Let us choose a basis $f_1,\dotsc, f_g$ of $\Lambda$ and a dual basis
$f^*_1,\dotsc, f^*_g$ of $\Lambda^*$, so that $(f_i^*,f_j)=
\sigma_{ij}$.  Then the elements of the lattice $M=\Sym^2\Lambda$ are
integral homogeneous quadratic functions $q= \sum_{i\le j}
q_{ij}f_if_j$, $q_{ij}\in\bZ$, on $\Lambda^*$. These correspond to
symmetric \emph{half-integral} $g\times g$ matrices $A=(a_{ij})$,
which means that $a_{ii}\in\bZ$ and $a_{ij}\in\frac12\bZ$ for $i\ne
j$. Equivalently, $2A$ is the matrix of an even integral bilinear
form.

The elements of $N=\Gamma^2\Lambda^*$ are integral tensors 
$\sum b_{ij} f_i^*\otimes f_j^*$ symmetric under the involution 
$f_i^*\otimes f_j^* \mapsto f_j^*\otimes f_i^*$. Thus, $N$ can be
identified with the space of symmetric \emph{integral} matrices $B=(b_{ij})$,
$b_{ij}\in\bZ$. 

Both $M$ and $N$ can be considered as the lattices in the space of
real symmetric $g\times g$-matrices. They are dual with respect to the
inner product $(A,B)=\trace AB$.

Now let $C$ be the open cone $C$ in $N_{\bR}$ consisting of
positive-definite symmetric real matrices. This cone is self-dual with
respect to the above inner product.  One fixes its ``closure''
$\oC$. To be precise, $\oC$ is the real cone spanned by semi definite
positive symmetric matrices $B\ge 0$ \emph{with rational radical}
(i.e., the null space of $B$ has to have a basis of vectors with
rational coordinates).

Then a toroidal compactification $\oA_g^{\tau}$ of $\A_g$ is defined
by a fan $\tau$ (i.e. a collection of finitely generated rational
cones, closed under taking faces)
in $N_{\bR}$ satisfying the following properties:
\begin{enumerate}
\item $\Supp\tau = \oC$.
\item The natural $\GL(g,\bZ)$-action on $N_{\bR}$ sends cones of
  $\tau$ to cones of $\tau$.
\item There are only finitely many orbits of cones under this action.
\end{enumerate}

The following are three standard fans corresponding to three standard
toroidal compactifications of $\A_g$:

\medskip 

{\bf The perfect cones fan} $\tau\perf$, otherwise known as the
\emph{1st Voronoi} fan. The cones are defined to be the
cones over the faces of the convex hull of $N\cap (\oC\setminus
0)$. By a result of Barnes and Cohn \cite{BarnesCohn}, the vertices of
$\Conv N\cap (\oC\setminus 0)$ (that is, the rays of $\tau\perf$) are
of the form ${a^*}^2$, where $a^* = \sum a_i f_i^*$ is an integral
primitive (i.e. indivisible) nonzero element of $\Lambda^*$. Thus,
every perfect cone $\sigma$ has the form $\sigma = \sum_s \bR_{\ge0}\,
{a^*_s}^2$ for some collection $\{a^*_s\} \subset\Lambda^*\setminus
0$.

For $q\in\Sym^2\Lambda_{\bR}$, one has $(q,{a^*}^2) = q(a^*)$, the value of the
quadratic function $q$ at the integral point
$a^*\in\Lambda^*$. Thus, if $\sigma^\vee$ is the dual cone in
$M_{\bR}$, then the elements of the interior $(\sigma^\vee)^0$ are the
positive definite quadratic functions which attain the minimum on the
same finite subset $\{a^*_s\}$.

In particular, for a maximal cone $\sigma\in\tau\perf$, the cone
$\sigma^{\vee}$ is generated by one quadratic function which is
determined up to a multiple by the set of its minimal
integral nonzero vectors.  Such quadratic forms are called \emph{perfect},
hence the name of this fan.

\medskip

{\bf The second Voronoi fan} $\tau\vor$, sometimes referred to as
Delaunay-Voronoi fan, or $L$-type decomposition. The locally closed
cones $\tau^0$ of this fan consist of quadratic forms which define the
same Delaunay decomposition of $\Lambda_{\bR}/ \Lambda$. 

\medskip

{\bf The central cones fan} $\tau\cent$, corresponding to the
normalization of the Igusa blowup.  Let $Q$ be the convex hull of
$C\cap (M\setminus 0)$. This is an infinite polyhedron whose faces are
(finite) polytopes.  The fan $\tau\cent$ is the dual fan of $Q$. The
vertices of $\Conv\big( C\cap (M\setminus 0)\big)$ are called
\emph{central quadratic forms}. Note that they are integral by
definition.  The corresponding cones of $\tau\cent$ are
maximal-dimensional \emph{central cones}.

Each of the fans $\tau\perf$, $\tau\vor$, $\tau\cent$ admits a
strictly convex support function, (1) and (3) by definition and (2) by
\cite{Alexeev_CMAV}. Hence, the compactifications $\oA_g\perf$,
$\oA_g\vor$, $\oA_g\cent$ are projective by Tai's criterion 
\cite[IV.2]{AMRT}.

\section{Graphs and quadratic forms}
\label{sec:graphs}

$G$ will denote a graph with edges $e_i$, $i=1,\dotsc,m$ and vertices
$v_j$, $j=1,\dotsc,n$. 
We allow multiple edges and loops. We fix an
orientation of edges. Then we have the usual boundary homomorphism
\begin{equation*}
  \partial\colon 
  C_1(G,\bZ) = \oplus_{i} \bZ e_i 
  \to C_0(G,\bZ) = \oplus_j \bZ v_j 
  ,\qquad
  \partial e_i = \operatorname{end}(e_i) - \operatorname{beg}(e_i)
\end{equation*}
The kernel of this map is the space of cycles $H_1(G,\bZ)$ and the
cokernel is $H_0(G,\bZ)$. We will assume $G$ to be
connected, so that $H_0(G,\bZ)=\bZ$. Dually, we have the homomorphism
\begin{equation*}
  d\colon 
  C^0(G,\bZ) = \oplus_j \bZ v_j^*
  \to C^1(G,\bZ) = \oplus_{i} \bZ e_i^*
  ,\qquad
  d v_j^* = 
  \sum_{ v_j = \operatorname{end}(e_i)} e_i^* - 
  \sum_{ v_j =  \operatorname{beg}(e_i)} e_i^*
\end{equation*}
with kernel $H^0(G,\bZ)=\bZ$ and cokernel $H^1(G,\bZ)$. 

\begin{definition}
  We call the elements $e_i^*$ in $H^1(G,\bZ)$ \emph{coedges}, to
  distinguish them from the edges $e_i\in C_1(G,\bZ)$. Thus, coedges
  are cocycles and edges are chains.
\end{definition}

Since any graph is homotopy equivalent to a graph with one vertex and
$g$ loops for some $g\ge0$, called the \emph{genus} of $G$,
$H_1(G,\bZ)$ and $H^1(G,\bZ)$ are free abelian groups of
rank $g$, dual to each other.  

\begin{lemma}\label{1conn-cohomology}
  One has the following:
  \begin{enumerate}
  \item The elements $e_i^*$ span $H^1(G,\bZ)$.
  \item $e_i^*=0$ iff the edge $e_i$ is a bridge in $G$.
  \item The graph $G$ is a simple loop (a graph with one vertex and
    one edge) or is loopless and is 2-connected $\iff$ the edges
    \emph{can not} be divided into two disjoint groups $I_1\sqcup I_2$
    such that
    \[
    H^1(G,\bZ) = \langle e^*_{i_1} \rangle
    \oplus \langle e^*_{i_2}\rangle,
    \quad i_s\in I_s.
    \]
  \end{enumerate}
\end{lemma}
\begin{proof}
  (1) and (2) are obvious. For (3), consider a partition of edges
  $I_1\sqcup I_2$, and denote by $G_s$, $s=1,2$, the graph formed by
  the edges of $I_s$. Note that the zero set of $I_s$ in $H_1(G)$ is
  $H_1(G_{3-s})$. Since $e_i^*$ span $H^1(G,\bZ)$, the condition of
  (3) is that the intersection is zero, equivalently that $H_1(G)$ is
  spanned by $H_1(G_1)$ and $H_1(G_2)$, i.e. every simple cycle in $G$
  lies entirely either in $G_1$ or in $G_2$. If $G$ has a loop (but
  $G$ is
  not a loop itself) or $G$ is not
  2-connected, then obviously there is such a decomposition. Vice
  versa, given such a decomposition, every vertex in $G_1\cap G_2$ is
  a cut of $G$ or is a vertex of a loop, so $G$ is not 2-connected or
  it has a loop.
\end{proof}

The following lemma gives explicit $\bZ$-bases for $H_1(G,\bZ)$ and
$H^1(G,\bZ)$.

\begin{lemma}\label{lem:coedges-unimodular}
  For a collection of edges $e_i$, $i\in I\subset \{1,\dots,m\}$, the
  following conditions are equivalent:
  \begin{enumerate}
    \item $e_i^*$ form an $\bR$-basis of $H^1(G,\bR)$.
    \item $e_i^*$ form a $\bZ$-basis of $H^1(G,\bZ)$.
    \item The complement of $\{e_i\}$ is a spanning tree $T$ of $G$.
  \end{enumerate}
  If either of these conditions is satisfied then there exists a basis
  of $H_1(G,\bZ)$ of the form
  \begin{equation*}
    f_i = e_i + \sum_{e_s\in T} b_{is} e_s, 
    \quad b_{is} = 0, \pm 1, \ i\in I.
  \end{equation*}
\end{lemma}
\begin{proof}
  Of course, (2) implies (1). Let us prove (1)$\Rightarrow$(3). Note
  that $|I|=g$.

  By the Euler's formula, $g(G)= m+1-n$ and $\chi(G) = 1-g$.
  Since the graph $G'=G\setminus \{e_i,\ i\in I\}$ has the same
  vertices and $g$ fewer edges, we have $\chi(G')=1$. Then either
  $G'$ is connected and is a tree, or else $G'$ is
  disconnected and has a nonzero loop, call it $\ell$. Then for all $i\in I$
  we have $e_i^*(\ell)=0$, hence $\{e_i^*,\ i\in I\}$ is not a basis
  of $H^1(G,\bR)$. QED.

  (3)$\Rightarrow$(2). We prove this by constructing a dual basis
  $\{f_i\}$ in $H_1(G,\bZ)$ to the set $\{e_i^*\}$.  Since $T$
  is a tree, for each $j$ there exists a unique path in $T$ from the
  end to the beginning of $e_i$. In other words, there exists a unique
  $f_i\in H_1(G,\bZ)$ which can be written as
  \begin{equation*}
    f_i = e_i + \sum_{e_s\in T} b_{is} e_s, 
    \quad b_{is} = 0, \pm 1.
  \end{equation*}
  Then it is clear that   $e_i^*(f_k) = 1$ if $j=k$ and $0$
  otherwise. Thus, $\{e_i^*\}$ and $\{f_i\}$ are dual bases in 
  $H^1(G,\bZ)$ and $H_1(G,\bZ)$.
\end{proof}

\begin{definition}
  An \emph{edge-minimizing metric}, abbreviated to \emph{emm}, of a graph
  $G$ is a quadratic form $q\in \Sym^2H_1(G)$ such that
  \begin{enumerate}
  \item $q>0$, i.e. $q$ is positive definite.
  \item $q(e_i^*)=1$ for each edge $e_i$ which is not a bridge
    (i.e. for each $e_i^*\ne0$).
  \item $q(v^*)\ge1$ for any $v^*\in H^1(G,\bZ)\setminus 0$.
  \end{enumerate}
  A \emph{strong} edge-minimizing metric, in addition, satisfies the
  following: if $q(v^*)=1$ for some $v^*\in H^1(G,\bZ)$ then $\pm v^*$
  is a coedge.

  In other words, $q$ is a metric on the lattice $H^1(G,\bZ)$ and the
  nonzero $\pm e_i^*$ are among the shortest 
  (resp. exactly the shortest) integral vectors in this metric.
\end{definition}

We will distinguish between $q\in \Sym^2H_1(G,R)$, where $R$ is $\bZ$,
$\bQ$, or $\bR$. We will call these $\bZ$-emm, $\bQ$-emm, $\bR$-emm
respectively. There will be no difference between $\bQ$-emms and
$\bR$-emms for our purposes.

\begin{definition}
  By Lemma~\ref{1conn-cohomology}(3), one has $H^1(G,\bZ) = \oplus_k
  H^1(G_k,\bZ)$ for some graphs $G_k$ so that each $G_k$ is either
  a simple loop or loopless and 2-connected, and so that each nonzero
  $e_i^*$ lies in one of the direct summands. We may call $G_k$
  \emph{irreducible components of~$G$}.
\end{definition}

\begin{lemma}
  There exists a $(\bZ,\bQ,$ or $\bR)$ emm for a graph $G$ $\iff$
  there exist emms for each irreducible component $G_k$.
\end{lemma}
\begin{proof}
  The restriction $q_k$ of an emm $q$ to each $H^1(G_k,\bZ)$ is an emm. Vice
  versa, given emms $q_k$ for graphs $G_k$, we can take $q=\sum q_k$
  to be an emm for $G$.
\end{proof}

\begin{lemma}\label{lem:reduction}
  To construct a $(\bZ,\bQ,$ or $\bR)$ emm for a graph $G$, it is
  sufficient to construct an emm for several related cubic bridgeless
  graphs.
\end{lemma}
A remark concerning our terminology: \emph{cubic} is the same as trivalent, and
\emph{bridgeless} is the same as 2-connected.
\begin{proof}
  By the above Lemma, it is sufficient to construct an emm for
  each irreducible component $G_k$. If $G_k$ is a loop then $q=x^2$ is
  a $\bZ$-emm. So assume $G_k$ is not a loop.
  
  Removing vertices of degree 2 and replacing the adjacent two edges by
  a single edge results in reducing some duplication in the set $\{
  e_i^* \}$.  Next, we inductively insert an edge into a vertex of
  degree $\ge4$, until we get to a cubic graph $G'_k$. Doing so does
  not change $H_1$ but adds more vectors $e_i^*$, so the condition for
  $G'_k$ is stronger than for $G_k$. 
\end{proof}

%%%%%%%%%%%%%%%%%%%%%%%%%%%%%%%%%%%%%%%%%%%%%%%%%%
\section{Criteria for the regularity of the extended Torelli map}
\label{sec:extension-general}

As in Section~\ref{sec:toroidal-comp}, we fix a lattice
$\Lambda\simeq\bZ^g$, a fan $\tau$, and a corresponding toroidal
compactification $\oA_g\utau$.  We also consider a graph $G$ of genus
$a\le g$ and write its homology as a quotient $\Lambda\onto
H_1(G,\bZ)$.  This gives a cotorsion embedding of $N(G):=\Gamma^2
H^1(\Gamma,\bZ)$ into $N=\Gamma^2\Lambda^*$.
We work either over a field or over $\bZ$.

\begin{definition}\label{def:set-squares}
  We will denote by $S(G)$ the set of nonzero vectors ${e_i^*}^2$ in
  $N(G)$.
\end{definition}

\begin{theorem}[General criterion]
  \label{thm:gen-crit}
  The Torelli map $\oM_g\dashrightarrow \oA_g\utau$ is regular
  in a neighborhood of a stable curve $[C]$ iff for the dual graph $G(C)$
  there exists a cone $\sigma$ in the fan $\tau$ such that
  $S(G)\subset\sigma$.
\end{theorem}
\begin{proof}
  The stacks $(\oM_g,\partial\oM_g)$ and
  $(\oA_g\utau,\partial\oA_g\utau)$ are toroidal. The second one by
  definition, and the first one because it is a smooth stack of
  dimension $3g-3$ and the boundary divisors have normal crossings.

  Thus, in a neighborhood of a boundary point $[C]$ of $\oM_g$
  corresponding to a stable curve, $\oM_g$ is a toroidal stack
  modeled on $(\bA^1,0)^m \times \bG_m^{3g-3-m}$, where $m$ is the
  the number of edges of the dual graph $\Gamma$ of $C$, and $\bG_m$
  is the multiplicative group. This corresponds to a standard
  $m$-dimensional cone in $\bR^{3g-3}$ generated by the first $m$
  coordinate vectors which are in a bijection with the edges $e_i$ of
  $\Gamma$.

  By the Picard-Lefshetz monodromy formula, near the boundary the
  Torelli map is described by the linear map sending the vector $e_i$
  to $(e_i^*)^2\in N$. We conclude the proof by applying a well-known
  criterion of regularity for toroidal varieties saying that the
  rational map is regular iff every cone of the first fan maps into a
  cone in the second fan.

  The coarse moduli spaces are locally finite Galois quotients of
  appropriate toroidal neighborhoods for $\oM_g$, $\oA_g\vor$. The
  regularity of the rational map is unaffected by such Galois covers. Hence, 
  the result for the coarse moduli spaces is the same as for the stacks.
\end{proof}

\begin{lemma}\label{lem:regular-open}
  The map $\oM_g\dasharrow\oA_g\utau$ is regular on an open union of
  strata of~$\oM_g$.
\end{lemma}
\begin{proof}
  For two stable curves $C,C'$, the stratum of $C$ is in the closure
  of the stratum of $C'$ iff $G$ is a contraction of $G'$. Then
  $H_1(G',\bZ)\onto H_1(G,\bZ)$, the lattice $N(G)$ is a cotorsion
  sublattice in $N(G')$, and $S(G)\subset S(G')$.  Thus, if
  $S(G')\subset \sigma$ then $S(G)\subset \sigma$. $S(G)$ also lies in
  a cone of the induced fan on $N(G)_{\bR}$.
\end{proof}

As a first application, we reprove the following result of
Mumford and Namikawa, cf. \cite[\S18]{Namikawa_NewCompBoth}.

\begin{theorem}  \label{thm:crit-2nd-vor}
  The map $\oM_g\to\oA_g\vor$ is regular. 
\end{theorem}
\begin{proof}
  This immediately follows from Theorem~\ref{thm:gen-crit},
  Lemma~\ref{lem:coedges-unimodular}, and the following well known
  elementary fact about the ``dicing'' 2nd Voronoi cones,
  cf. \cite{ErdahlRyshkov_OnLatticeDicing}.
\end{proof}

\begin{lemma}[Dicings]
  Let $v^*_i\in\Lambda^*$, $i\in I$, be finitely many nonzero primitive
  vectors. 
  Then the following conditions are equivalent:
  \begin{enumerate}
  \item $\{ {v^*_i}^2 \}$ lie in the same 2nd Voronoi cone,
  \item $\sum \bR_{\ge0}\, {v^*_i}^2$ is a 2nd Voronoi cone,
  \item Any linearly independent subset $\{ {v^*_{j}} \}$, $j\in
    J\subset I$, is a $\bZ$-basis of $\Lambda^*\cap \sum \bR {v^*_{j}}
    $.
  \end{enumerate}
\end{lemma}

\begin{remark}\label{rem:toroidal-Schottky}
  The systems of vectors in the above lemma are known by various
  names: totally unimodular systems of vectors, dicings, regular
  matroids. The ``dicing'' refers to the corresponding Delaunay
  decomposition of $\Lambda_{\bR}$ periodic w.r.t. $\Lambda$. It is
  given by ``dicing'' the vector space $\Lambda_{\bR}$ by the parallel
  systems of hyperplanes $\{v_i^*=n_i\in\bZ\}$.

  Seymour's classification theorem on regular matroids, which can be
  found in \cite{Oxley_Matroids}, says that all regular matroids are
  graphic, cographic, a special matroid $R_{10}$, or can be obtained
  from these by a sort of ``tensor product''.

  The regular matroids above, corresponding to $\{ e_i^* \in
  H^1(G,\bZ) \}$, are the cographic matroids. This gives the
  combinatorial description of the toroidal Torelli map. By
  \cite[4.1]{ErdahlRyshkov_OnLatticeDicing} every dicing 2nd Voronoi
  cone is simplicial. Thus, the open neighborhood of $\im\oM_g$ in
  the stack $\oA_g\vor$ corresponding to the cographic dicing cones
  has at worst abelian quotient singularities. 
\end{remark}

We now turn to the cases of perfect and central cones.

\begin{theorem}\label{thm:crit-perf}
  \begin{enumerate}
  \item The Torelli map $\M_g \dashrightarrow \oA_g\perf$ is regular
    in a neighborhood of a stable curve $[C]$ iff the dual graph
    $G(C)$ has an $\bR$-emm.
  \item Moreover, if every graph $G$ of genus $\le g$ has a strong
    $\bR$-emm then $\oA_g\perf$ and $\oA_g\vor$ share a common open
    neighborhood of the image of $\oM_g$.
  \end{enumerate}
\end{theorem}
\begin{proof}
  (1) By the description of the perfect fan given in
  Section~\ref{sec:toroidal-comp}, ${e_i^*}^2$ lie in a perfect cone
  iff they are edges of some perfect cone $\sigma$. This means that
  there exists a positive definite quadratic form $q$ such that the
  nonzero ${e_i^*}$ are among the shortest integral vectors
  w.r.t. $q$. This is our definition of an $\bR$-emm.

  (2) By the above, the strata of
  $\oA_g\vor$ corresponding to the cographic regular matroids gives a
  Zariski open neighborhood $U$ of the image of $\oM_g$. We want to
  show that each of these 2nd Voronoi cones is also a perfect
  cone. This means that there exists a $q>0$ such that the nonzero
  $\pm e_i^*$ are exactly the shortest integral vectors
  w.r.t. $q$. This is our definition of a strong $\bR$-emm.
\end{proof}

\begin{theorem}\label{thm:crit-cent}
  The Torelli map $\M_g \dashrightarrow \oA_g\cent$ is regular in a
  neighborhood of a stable curve $[C]$ iff the dual graph $G(C)$ has
  a $\bZ$-emm.
\end{theorem}
\begin{proof}
  Applying Theorem~\ref{thm:gen-crit}, if the map is regular then $\{
  {e_i^*}^2 \}$ lie in the same central cone $\sigma$, which we can
  pick to be maximal-dimensional. The corresponding dual cone
  $\sigma^{\vee}$ is 1-dimensional and is spanned by a central form
  $q\in M$. This is an integral positive definite form characterized
  by the following property: for any $f\in\sigma$ and any other
  integral positive definite form $q'\in M$ one has $(q,f) \le
  (q',f)$.  Since every $e_i^*\ne0$ is a primitive vector in
  $\Lambda^*$, there exist a $q'$ with $(q',{e_i^*}^2) =
  q'(e_i^*)=1$. Therefore, $q(e_i^*)=1$ for all nonbridge edges $e_i$,
  and so $q$ is a $\bZ$-emm.

  Vice versa, if $q$ is a $\bZ$-emm of $G$ then $1=q(e_i^*) \le
  q'(e_i^*)$ for any $q'\in M\cap C\setminus 0$, so $\{ {e_i^*}^2 \}
  \subset \sigma$ for any cone $\sigma^{\vee}$ containing $q$.
\end{proof}

\begin{lemma}
  The subset of $\oM_g$ of curves admitting a $\bZ$-emm,
  resp. $\bR$-emm, is an open union of strata
  (cf. Lemma~\ref{lem:reduction}).
\end{lemma}
\begin{proof}
  Using the notations of the proof of Lemma~\ref{lem:regular-open},
  the restriction of an emm on $H^1(G',\bZ)$ to $H^1(G,\bZ)$ is an emm
  for $G$.
\end{proof}

\section{$\bZ$-emms and positive cycle 2-covers of graphs}
\label{sec:Z-emms}

\begin{lemma}
  Let $q$ be a $\bZ$-emm of a graph $G$. Then the lattice
  $(H^1(G,\bZ),2q)$ is a direct sum of the standard root lattices $A_n$,
  $D_n$ ($n\ge4)$, $E_n$ ($n=6,7,8$).

  If, in addition, $G$ has no loops and is 2-connected, then the root
  lattice $(H^1(G,\bZ),2q)$ is irreducible, i.e. it is a single copy
  of $A_n$, $D_n$, or $E_n$.
\end{lemma}
\begin{proof}
  The bilinear form $2q$ is integral, symmetric, positive definite,
  and even. The set $R$ of vectors $r$ with $r^2=2$ spans $H^1(G,\bZ)$
  since it contains $e_i^*$. Thus, $R$ is a simply laced root system,
  which must be a direct sum of $A_n,D_n,E_n$ by a standard
  classification.

  If $(H^1(G,\bZ),2q)$ is not irreducible then it splits as a direct
  sum $\langle e^*_{i_1} \rangle \oplus \langle e^*_{i_2}\rangle$ for
  some partition $I_1\sqcup I_2$ of edges. Then $G$ is not 2-connected
  loopless by Lemma~\ref{1conn-cohomology}.
\end{proof}

Recall that a graph is called \emph{projective planar} if it admits
an an embedding into $\bR\bP^2$.  
The main result of this section is the following.

\begin{theorem}\label{thm:AnDn}
  Let $G$ be a 2-connected loopless graph. Then
  \begin{enumerate}
  \item $G$ has a $\bZ$-emm of type $A_g$ $\iff$ $G$ is planar.
  \item For $g\ge4$, $G$ has a $\bZ$-emm of type $D_g$ $\iff$ $G$ is
    projective planar.
  \end{enumerate}
  Moreover, the directions $\Leftarrow$ hold for any graph.
\end{theorem}

\begin{remark}
  The direction $\Leftarrow$ of (1) is due to Namikawa
  \cite[Prop.5]{Namikawa_ExtendedTorelli}.
\end{remark}

\begin{corollary}
  Every graph of genus $g\le 6$ admits a $\bZ$-emm.
\end{corollary}
\begin{proof}
  The well-known Kuratowski theorem says that a graph is not planar
  iff it contains a subgraph homeomorphic to $K_5$ or $K_{3,3}$. There
  is a similar theorem of Archdeacon 
  \cite{Archdeacon_Thesis, Archdeacon_ProjPlane}
  for projective planar graphs which has a much
  longer list of 103 minimal counterexamples. All of those graphs have
  genus $\ge7$, except for a single graph $G_1$. The existence of a
  $\bZ$-emm for $G_1$ can be easily established by a direct, although
  quite lengthy, computation.

  We note that 
  \cite{Archdeacon_Thesis, Archdeacon_ProjPlane}
  is concerned with graphs without loops and and multiple edges. But a
  loop just adds a single $\bZ$ summand to $H_1(G,\bZ)$, and the
  multiple edges do not affect (projective) planarity.
\end{proof}

\begin{remark}
  By extending this method, \cite{Vigre_Genus7} proves the existence
  of a $\bZ$-emm for every graph of genus 7 and 8. This amounts to
  checking all the cubic genus 6 and 7 graphs from Archdeacon's list
  and the cubic graphs obtained from them by adding one or two edges.
\end{remark}

\begin{corollary}
  For any $g\ge9$, there exists a graph $G$ of genus $g$ which does
  not admit a $\bZ$-emm.
\end{corollary}
\begin{proof}
  Since $\bZ$-emms of type $E_n$ only appear for $g=6,7,8$, in genus $g\ge9$
  it is sufficient to take any 2-connected loopless graph containing a
  graph from Archdeacon's list
  \cite{Archdeacon_Thesis,Archdeacon_ProjPlane} as a subgraph. For
  example, the graph of genus 9 in Figure~\ref{fig:no-emm} contains
  the minimal nonplanar graph of genus 6 from
  \cite{Archdeacon_Thesis,Archdeacon_ProjPlane}:
\begin{figure}[h]
 \begin{tikzpicture}
  [scale=.7,auto=left]
%  \draw[help lines] (1,1) grid (12,7);

  \node[vert] (n1) at (1,4) {};
  \node[vert] (n2) at (3,6)  {};
  \node[vert] (n3) at (5,4)  {};
  \node[vert] (n4) at (3,2) {};
  \node[vert] (n5) at (3,4)  {};
  \node[vert] (n7) at (8,4) {};
  \node[vert] (n8) at (10,6) {};
  \node[vert] (n9) at (12,4) {};
  \node[vert] (n10) at (10,2) {};
  \node[vert] (n11) at (10,4) {};

  \node[vert] (v1) at (6.5,7) {};
  \node[vert] (v2) at (6.5,5) {};
  \node[vert] (v3) at (6.5,1) {};

  \draw[-](n1) to [out=90,in=180]  (n2);
  \draw[-] (n2) to [out=0,in=90]  (n3);
  \draw[-] (n3) to [out=-90,in=0]  (n4);
  \draw[-] (n4) to [out=180,in=-90]  (n1);

  \draw[-] (n7) to [out=90,in=180]  (n8);
  \draw[-] (n8) to [out=0,in=90]  (n9);
  \draw[-] (n9) to [out=-90,in=0]  (n10);
  \draw[-] (n10) to [out=180,in=-90]  (n7);

  \draw[-] (n1) --  (n5);
  \draw[-] (n5) --  (n3);
  \draw[-] (n7) --  (n11);
  \draw[-] (n11) -- (n9);

  \draw[-] (n2) to [out=30,in=180] (v1) [out=0,in=150] to (n8);
  \draw[-] (n5) to [out=30,in=180] (v2) [out=0,in=150] to (n11);
  \draw[-] (n4) to [out=-30,in=180] (v3) [out=0,in=-150] to (n10);

  \draw[-] (v1) to (v2) to (v3) to [out=60,in=-60] (v1);
\end{tikzpicture}
\caption{A graph of genus 9 without a $\bZ$-emm}
\label{fig:no-emm}
\end{figure}
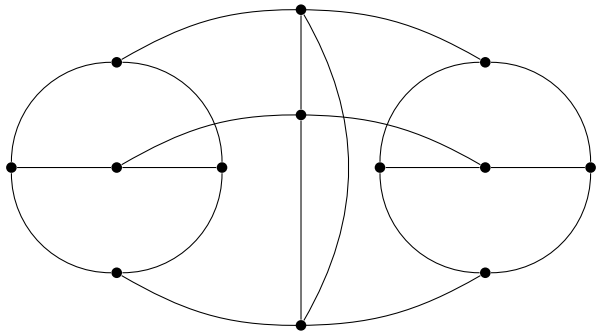
  For $g>9$, one can obtain $G$ from it by adding $g-9$ edges.
\end{proof}

Joe Tennini pointed out to us that the list in
\cite{Archdeacon_Thesis,Archdeacon_ProjPlane} contains a graph with 7
vertices. This implies that the complete graph $K_n$ is not projective
planar for $n\ge7$ and does not admit a $\bZ$-emm since $g(K_7)=15$.

  \bigbreak

The basic idea of our proof of Theorem~\ref{thm:AnDn} is the
following. We can assume that $G$ is bridgeless, by contracting the
bridges. A \emph{cycle 2-cover} of $G$ is a collection of cycles $c_k$
such that every edge appears in $\cup c_k$ exactly twice.  (A long
standing conjecture of Szekeres-Seymour says that every bridgeless
graph has such a cycle 2-cover. We don't need the validity of this
conjecture for our proof).

Now let us say $\{ c_k \}$ is a cycle 2-cover, and consider the
quadratic form $q=\frac12 \sum c_k^2$. Then $q$ is integral and since
every $e_i^2$ appears in $q$ with coefficient 1, one has $q(e_i^*) =
(q, {e_i^*}^2 ) =1$. However, in general $q$ is only positive \emph{semi}
definite. 

\begin{definition}
  A cycle 2-cover $\{c_k,\ k=1,\cdots N\}$ of a graph is called
  \emph{positive} if the quadratic form $q=\frac12\sum_{k=1}^N c_k^2$ is
  positive definite.
\end{definition}

Cycle 2-covers are closely related to embeddings of graphs into closed
topological surfaces. 
Given an embedding $G\into S^2$, resp. $G\into\bR\bP^2$, we will
construct a $\bZ$-emm of type $A_g$, resp. $D_g$, on $H^1(G,\bZ)$.
Then we will prove the converse by
using the fact that the quadratic forms $A_n$ and $D_n$ can be written
as sums of $\ge n$ squares of integral linear forms.

So let $\{c_k\}$ be a cycle 2-cover. Divide each $c_k$ into a sum
of simple (not repeating vertices) cycles $\dl$.  For each $\dl$,
take a copy $\Dl$ of a 2-disk and identify its boundary with $\dl$. 
Glue these disks along the edges $e_i$. The result is a closed surface
$X$ which may have isolated singular points at some
vertices, as follows.

For a vertex $v$, consider a simple cycle $\dl$ 
   in the given 2-cover that goes through $v$. We constructed $X$ by
   gluing the boundary of a disk $D_{\ell}$ to $d_{\ell}$. In $X$ there
   is a neighboring disk that also contains $v$; continue from disk to
   disk until you have made a full circle around $v$. 
If these are \emph{all} the $\dl$ containing $v$ then $X$ is
smooth at $v$. In general, there will be several such full circles,
and so $X$ is obtained from a smooth closed surface $\wX$ by gluing
together several points to such bad vertices $v$.

\begin{theorem}\label{thm:positive-cover}
  Let $G$ be a loopless 2-connected graph admitting a cycle double
  cover $\{c_k\}$. Then $\{c_k\}$ is positive $\iff$ $X=S^2$ or
  $\bR\bP^2$.
\end{theorem}
\begin{proof}
  Assume that $q>0$. First, we claim that $X$ has no singular points,
  i.e. that $X=\wX$. Assuming the opposite, let $X'$ be the surface
  obtained from $X$ by normalizing at a single singular point $v$, and
  let 
  $G'$ be the preimage
  of $G$ in $X'$. Since $G$ is 2-connected, $G'$ is a connected
  graph. Since $G'$ has the same edges as $G$ but more vertices,
  we have $g(G')<g(G)$, and $H_1(G')\subset H_1(G)$ is a proper
  subspace. But $q$ is the same sum of squares of elements of $H_1(G')$,
  so it can not be positive definite.

  Next, the Euler characteristic of the smooth surface $X$ is 
  \[ \chi(X) = N - E + V = N- (E-V+1) + 1 = N-g +1 \]
  Since $q$ is positive definite, $N\ge g$. Hence, $\chi(X)\ge
  1$. There are only two smooth closed surfaces with $\chi\ge1$: $S^2$
  ($\chi=2$, $N=g+1$) and $\bR\bP^2$ ($\chi=1$, $N=g$).

  Now, for the opposite direction. If $X=S^2$ then $G$ divides the
  sphere into $g+1$ regions with boundaries 
  $c_k$. These obviously generate $H_1(G,\bZ)$, so
  $q>0$. 

  In the case $X=\bR\bP^2$, let $\pi:S^2\to\bR\bP^2$ be the 2:1 cover,
  and let $G'=\pi\inv(G)$. 
  Because $\pi_*H_1(G',\bR)= H_1(G,\bR)$ and the cycles on $S^2$
  generate $H_1(G',\bR)$, the cycles $c_k$ generate
  $H_1(G,\bR)$. Hence, $\sum c_k^2$ is positive definite. 
\end{proof}

\begin{proof}[Proof of Theorem~\ref{thm:AnDn}]
  (1) Let $G$ be an arbitrary graph with an embedding $G\subset S^2$.
  If $G$ is not connected, we add bridges to make it connected. Then
  the set $S^2\setminus G$ is a union of $g+1$ regions bounded by the
  cycles $c_k$ which can be given compatible orientations.  Then
  $H_1(G,\bZ)= \bZ^{g+1}/\bZ\sum c_k$, and the dual lattice
  $H^1(G,\bZ)$ is the hyperplane $\{ (n_k) \in \bZ^{g+1} \mid \sum n_k
  =0 \}$.  The quadratic form $2q=\sum c_k^2$ is the restriction of
  the standard Euclidean form on $\bZ^{g+1}$ to $H^1(G,\bZ)$. This is
  the standard definition of the $A_g$ lattice.  Each nonbridge edge
  $e_i$ belongs to precisely two 2-cells $k_1,k_2$. Then
  $\pm e_i^*=c_{k_1}^*-c_{k_2}^*$, where $\{c_k^*\}$ is the Euclidean
  basis of $\bZ^{g+1}$.

  Vice versa, suppose that a 2-connected loopless graph $G$ has a
  $\bZ$-emm of type $A_g$.  Note that the quadratic form of $A_g$ is a
  sum of $g+1$ squares of integral linear functions:
  \begin{eqnarray*}
    2q = 2 \sum_{i=1}^g x_i^2 - 2 \sum_{i=1}^{g-1} x_i x_{i+1} =
    x_1^2 + (x_1-x_2)^2 + \dots (x_{g-1}-x_g)^2 + x_g^2
  \end{eqnarray*}
  Consider the $g+1$ cycles $c_k$ in $H_1(G,\bZ)$ corresponding to
  these linear terms. We claim that for each edge $e_i$ of $G$ we have
  $c_k(e_i^*)=0$ or $\pm1$. Indeed, if $|c_k(e_i^*)|\ge 2$ then for
  $q=\frac12\sum_k c_k^2$ we have $q(e_i^*)\ge 2$, which contradicts
  our assumption $q(e_i^*)=~1$.  Thus, for each edge there exist
  exactly two cycles with $c_k^2(e_i^*)=1$, the collection of $c_k$ is
  a cycle 2-cover.  By the proof of the previous
  Theorem~\ref{thm:positive-cover} the corresponding ambient surface
  is $S^2$.

  (2) Let $G\subset \bR\bP^2$ be an arbitrary graph. Again, we make it
  connected, if necessary, by adding bridges.  Let $G'\subset S^2$ be
  the preimage under the 2:1 cover $S^2\to\bR\bP^2$.  

  If $G'$ is disconnected then it has two components both isomorphic
  to $G$. Then $G$ is planar, (1) applies, and an embedding $G\into
  S^2$ defines a $\bZ$-emm of type $A_g$. Consider the $g+1$ regions
  $S^2\setminus G$ and the corresponding cycles $c_k$. Since $g+1\ge
  5$ and the complete graph $K_5$ is non-planar, there exist two
  cycles $c_{k_1}$, $c_{k_2}$ which do not share an edge. Then the
  quadratic form $\frac12\sum_{k\ne k_1,k_2} c_k^2 + \frac12(c_{k_1}- c_{k_2})^2$
  is a $\bZ$-emm of $G$, and it is easy to check that it has type $D_g$.

  Otherwise, $G'$ is a connected graph of genus $2g-1$. 
  As above, $H^1(G',\bZ)$ together with the quadratic form
  $\sum_{k=1}^{2g}c_k^2$ is a root system of type $A_{2g-1}$ in its
  realization as the hyperplane $\{ \sum n_k=0 \}$ in $\bZ^{2g}$.

  On the homology, the antipodal involution $\iota$ can be written as
  $c_k \mapsto -c_{k+g}$, where we set $c_m := c_{m-2g}$ if
  $m>2g$. The image of $H_1(G',\bZ)$ in $H_1(G,\bZ)$ is the projection
  of $H_1(G',\bZ)$ onto the $(+1)$-eigenspace $H_1^+(G',\bR)$, and can
  be identified with the standard Euclidean $\bZ^g$ with the basis
  $c_1,\dots,c_g$. The homology group $H_1(G,\bZ)$ is the
  $\bZ_2$-extension of it obtained by adding vector
  $\frac12(1,\dots,1)$.

  Thus, the dual lattice $(H^1(G,\bZ),2q)$ can be identified with the
  sublattice of $\bZ^g$ of integral vectors with even sum of
  coordinates. This is the standard definition of the $D_g$ lattice.

  For the opposite direction, note that the quadratic form for $D_g$
  is a sum of $g$ squares:
  \begin{eqnarray*}
    2q &=& 2 \sum_{i=1}^g x_i^2 
    -2x_1x_3 - 2x_2x_3 
    - 2 \sum_{i=3}^{g-1} x_i x_{i+1} =\\
    &=& 
    (x_1+x_2-x_3)^2 + (x_1-x_2)^2
    + (x_3-x_4)^2 + \dots (x_{g-1}-x_g)^2 + x_g^2
  \end{eqnarray*}
  This gives $g$ cycles $c_k$ and thus a positive double cover by $g$
  or $g+1$ simple cycles $\dl$, which defines an embedding of $G$ into
  $S^2$ or $\bR\bP^2$, if $G$ is 2-connected and loopless.  Finally,
  if $G$ is planar then it is moreover projective planar.
\end{proof}

\section{Existence of $\bR$-emms}
\label{sec:R-emms}

Recall that by Lemma~\ref{lem:reduction} it sufficient to construct
$\bR$-emms for cubic graphs.
We begin by characterizing coedges. The following simple lemma will be useful:

\begin{lemma}\label{cycle_lemma}
  If $G$ is a connected bridgeless graph with the property that every edge
  is contained in a two element cutset then $G$ is cyclic.
\end{lemma}

\begin{proof}
  Suppose, as the induction hypothesis, $G$ is a minimal counterexample. Take any
  edge $e$ and a cutset $\{e,f\}$ containing it. This exhibits $G$ as a "cycle"
  $[e,G_a, f, G_b]$ with $G_a$ and $G_b$ disjoint and joined only by $e$ and
  $f$. Enlarge this cycle to exhibit $G$ as a larger "cycle", or "necklace"
  consisting of a chain of "gems" $G_1, G_2, \dots G_r$ connected cyclically by
  single edges.  Further assume that this necklace is maximal, so that no
  $G_i$ contains a $G_i$-bridge. Then by the induction hypothesis, each
  $G_i$ must be cyclic. Since we assumed $G$ was not cyclic, one of
  them contains an edge. That edge is not in any 2-element $G$-cutset,
  contradicting the choice of $G$.
\end{proof}

We can now give the promised characterization of coedges:

\begin{definition}
  A cycle in a graph $\Gamma$ will be called a \emph{$(0,1)$-cycle} if
  all directed edges appear in it with coefficients in $\{+1, 0,
  -1\}$; that is, the cycle is a sum of simple cycles with disjoint
  edge supports.
\end{definition}

\begin{lemma}\label{coedge_lemma}
  A nonzero cocycle $z \in H^1(\Gamma,\bZ)$ is a coedge $\iff$ $z(c) \in
  \{+1, 0, -1\} $ for all $(0,1)$-cycles $c$.
\end{lemma}

\begin{proof}
  Clearly a coedge satisfies this condition, so we need to prove the
  converse. 

  We can assume that $\Gamma$ is 2-edge connected, i.e. connected and
  bridgeless. In a bridgeless graph, all edges are divided into equivalence
  classes by $e\sim e'$ iff $e^*=\pm {e'}^*$. By contracting all but one edge
  in each equivalence class, we can assume that $\Gamma$ is 3-edge connected.

  We will proceed by induction on the number of edges in
  $\Gamma$.  Choose some edge $e_1$ of $\Gamma$. We form a new graph
  $\Gamma\setminus e_1$ by deleting $e_1$.  
  We note that
  $(0,1)$-cycles in $\Gamma\setminus e_1$ are $(0,1)$-cycles in
  $\Gamma$, and we have the natural pullback map $f: H^1(\Gamma,\bR)
  \to H^1(\Gamma\setminus e_1,\bR)$, so $f(z)$ satisfies the
  conditions of the lemma and by induction is a coedge if it is nonzero. But
  $\ker(f)=e_1^*$, so either $z=ne_1^*$ (in which case we immediately have
  $z=\pm e_1^*$, a coedge) or $z=ne_1^*+e_2^*$. In this case we claim that $n
  \in \{-1, 0, 1\}$. To show this we exhibit a simple cycle $c$ in $\Gamma
  \setminus e_2$ containing $e_1$. This is certainly possible as long as $e_1$
  is not a bridge in $\Gamma \setminus e_2$. On the other hand, if $e_1$ is a
  bridge in $\Gamma \setminus e_2$ then $e_1^*= \pm e_2^*$ in $H^1(\Gamma,
  \bZ)$, so $f(z)=0$, and $z=\pm e_1^*$ as above.  If $n=0$ we're done, so can
  assume that $z=e_1^*+e_2^*$ by changing the orientation of $e_1$ if
  necessary.

  If $\Gamma\setminus \{e_1,e_2\}$ has a bridge $e_3$ then we get a
  3-term relation on coedges $e_3^*=e_1^* \pm e_2^*$ in $\Gamma$,
  implying either $z=e_3^*$, in which case we're done, or $z=2e_2^*-e_3^*$,
  which would contradict $z(c) \in \{+1, 0, -1\}$ for simple cycles $c$, by an
  argument similar to that given above. The alternative is that
  $\Gamma\setminus\{ e_1,e_2\}$ is bridgeless, which we show is impossible.
  
  Assume $\Gamma\setminus\{e_1,e_2\}$ is bridgeless. Delete any edge $e \neq
  e_1, e_2$. Induction tells us that $z$ becomes a coedge $e_3^*$ in
  $\Gamma\setminus e$, so we have a four term relation $e_1^*+e_2^* = e_3^* +
  ke$ in $\Gamma$. Thus we have a four element cutset $\{e_1, e_2, e, e_3\}$ on
  $\Gamma$ , and so a two element cutset $\{e, e_3\}$ on $\Gamma\setminus
  \{e_1, e_2\}$ for any $e$. But by the previous lemma~\ref{cycle_lemma}, the
  only graphs where every edge is contained in a two element cutset are cyclic
  graphs, and if $\Gamma\setminus \{e_1,e_2\}$ was a cyclic graph we could
  easily find a $(0,1)$-cycle $c$ in $\Gamma$ such that $(e_1^*+e_2^*)(c)=2$, a
  contradiction.  \end{proof}

\begin{theorem}\label{thm:Remm-cubic}
  Any bridgeless cubic graph $G$ admits a strong $\bQ$-emm.
\end{theorem}

\begin{proof}
  We will reduce the problem to the existence of strong $\bQ$-emms on certain
  strictly smaller graphs. In this way, we get an inductive
  construction of such forms. Let $e_0$ be an edge in $G$. We can
  produce 3 graphs on fewer vertices by modifying the region of $G$
  containing $e_0$ as shown on Figure~\ref{fig:1}.  Note that $G_1$
  and $G_2$ are analogous to each other. We call the two edges formed
  by this process $e_1$ and $e_2$ in each of $G_1$, $G_2$, $G_3$, and
  we orient them as shown.

\def\setup{
  \node[close] (sw) at (-1,-1) {};
  \node[close] (se) at ( 1,-1) {};
  \node[close] (nw) at (-1, 1) {};
  \node[close] (ne) at ( 1, 1) {};
  \node[close] (sc) at (0,-.5) {};
  \node[close] (nc) at (0, .5) {};
  \node[close] (cnw) at (-.1, .1) {};
  \node[close] (cse) at ( .1,-.1) {};
}
\def\makecircle{
  \draw[dashed,thin] (0,0) circle (1.414);
}

\begin{figure}[h]
  \centering
  \includegraphics{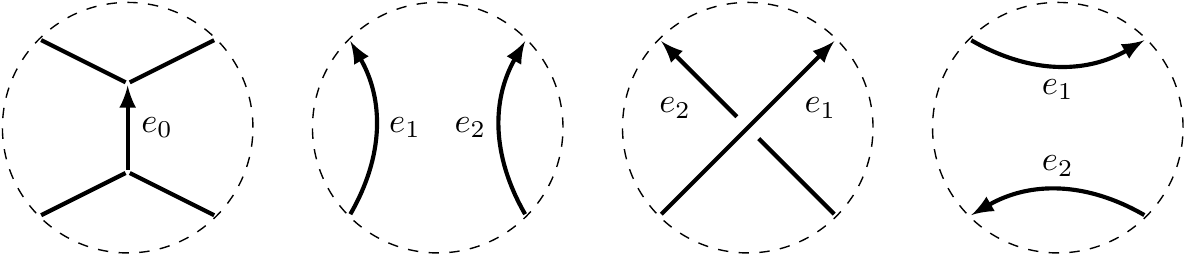}
  \caption{The graphs $G$, $G_1$, $G_2$, $G_3$}
  \label{fig:1}
\end{figure}

Since integral cycles in $G_i$ lift to integral cycles in $G$, we
have maps $H_1(G_i,\mathbb Z) \to H_1(G,\mathbb Z)$, $i
\in \{1,2,3\}$. Hence we get the opposite maps $\phi_i: H^1(G,\mathbb
Z) \to H^1(G_i,\mathbb Z)$.  We note the following:

\begin{claim}\label{cla:coedge}
  For a cocycle $z \in H^1(G,\mathbb Z)$, if $\phi_i(z)$ is a coedge
  for each $i \in \{1,2,3\}$ then $z$ is itself a coedge.
\end{claim}
\begin{proof}
  This follows immediately from the Lemma~\ref{coedge_lemma} and the
  observation that every $(0,1)$-cycle in $G$ corresponds to a $(0,1)$-cycle in
  at least one of the $G_i$
\end{proof}

Examining these maps more closely, we see that the kernels of $\phi_i$ are
generated by the cocycles shown in Figure~\ref{fig:2}.  We also have maps on
quadratic forms $\psi_i: \sym^2 H_1(G_i,\mathbb Z) \to \sym^2 H_1(G,\mathbb
Z)$. By the induction hypothesis we have strong $\bQ$-emms $q_i$ on $G_i$. These lift
to forms $\psi_i(q_i)$ on $G$, positive semidefinite and zero only on
$\ker(\phi_i)$. We wish to build a strong emm as a convex combination
$x_1\psi_1(q_1)+x_2\psi_2(q_2)+x_3\psi_3(q_3)$ where $x_1+x_2+x_3=1, x_i \ge
0$.

\begin{figure}[h]
  \centering
  \includegraphics{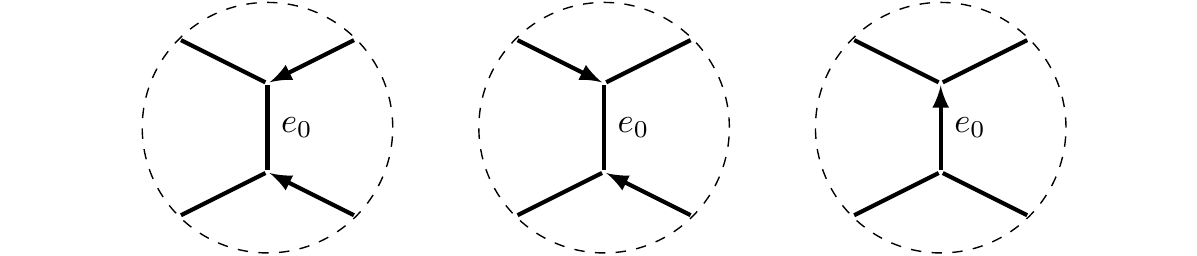}
  \caption{Generators of $\ker\phi_1$, $\ker\phi_2$, $\ker\phi_3$}
  \label{fig:2}
\end{figure}

Assume first that $G_i$ are bridgeless (the special cases where some of the
$G_i$ are not bridgeless will be dealt with later).  Note that for every edge
$e \neq e_0$ we have $\psi_i(q_i)(e^*)=1$ because $\phi_i$ maps coedges to
coedges. For every other (that is noncoedge) integral cocycle $z$ not in
$\ker(\phi_i)$ we have $\psi_i(q_i)(z) \ge 1$, with at least one of the
$\psi_i(q_i)(z) > 1$ by Claim~\ref{cla:coedge}.

Hence if $x_i \neq 0$ for all $i$, we need only to verify the emm conditions on
the three cocycles generating $\ker(\phi_i)$. For brevity we write
$c_i=q_i(e_1^*+e_2^*)$ and note that since $q_i$ is a quadratic form with
$q_i(e^*_1)=q_i(e^*_2)=1$ we must have $q_i(e_1^*-e_2^*)=4-c_i$ . Hence, if
$e^*_1 \pm e^*_2 \neq 0$ in $G_i$ then $c_i \in [1,3]$, with $c_i \in (1,3)$ if
$e^*_1 \pm e^*_2$ are not coedges. We can now write the (strong) emm conditions as follows:

\begin{enumerate}
\item $x_2(4-c_2)+x_3(4-c_3) \ge 1$ when $\ker(\phi_1) \neq \{0\}$ and
  with equality only when the generator of $\ker(\phi_1)$ is a coedge.
\item $x_1(4-c_1)+x_3c_3 \ge 1$ when $\ker(\phi_2) \neq \{0\}$ and
  with equality only when the generator of $\ker(\phi_2)$ is a coedge.
\item $x_1c_1+x_2c_2=1$ ($\ker(\phi_3)$ is generated by a coedge,
  namely $e_0^*$.)
\end{enumerate}

Note that these inequalities are symmetric except for the
symbols $\ge$, $=$. Consider the generic case where $c_1, c_2, c_3 \in (1,3)$
(we'll deal with the nongeneric cases later). We use three solution
types, depending on the values of $c_1$, $c_2$, $c_3$:

\begin{case}
  When $(c_2 \le 2 \myor c_3 \le 2) \myand (c_1 \le 2 \myor
  c_3 \ge 2)$ set $x_1=x_2=\dfrac1{c_1+c_2}$.  Consider the left hand side
  of the first inequality, call it $k$. We have
  $k=\dfrac{4-c_2}{c_1+c_2}+\dfrac{(c_1+c_2-2)(4-c_3)}{c_1+c_2}$. Since
  $c_3 < 3$ we have $k > \dfrac{2+c_1}{c_1+c_2} \ge 1$ when $c_2 \le
  2$.  If $c_3 \le 2$ then $k \ge \dfrac{2c_1+c_2}{c_1+c_2} > 1$. The
  other half of the conjunction follows by symmetry.
\end{case}

\begin{case}
  When $c_1 \ge 2 \myand c_3 < 2$, set $x_2=\epsilon c_1, x_1=1/c_1-\epsilon
  c_2$ for small $\epsilon>0$. Indeed, if we take $\epsilon = 0$ the left hand side of
  the first inequality simply becomes $\dfrac{(c_1-1)(4-c_3)}{c_1}$, which is
  greater than 1. The second inequality is also satisfied (see case 1). Of
  course this doesn't quite work, since $x_2=0$, but by continuity the
  conditions still hold for suitably small $\epsilon$.
\end{case}

\begin{case}
  When $c_2 \ge 2 \myand c_3 > 2$, set $x_1=\epsilon
  c_2, x_2=1/c_2-\epsilon c_1$. This works by symmetry with case 2.
\end{case}	

Hence all the generic cases are solved.
Consider now the exceptional cases. That is, assume that either at least one of
the $G_i$ has a bridge, or that one of the $c_i$ is 0, 1, 3, or 4. By listing all 
possible such cases, we will
show that they are of only three types (up to symmetry), shown in
Figure~\ref{fig:3}.

\begin{figure}[h]
  \centering
  \includegraphics{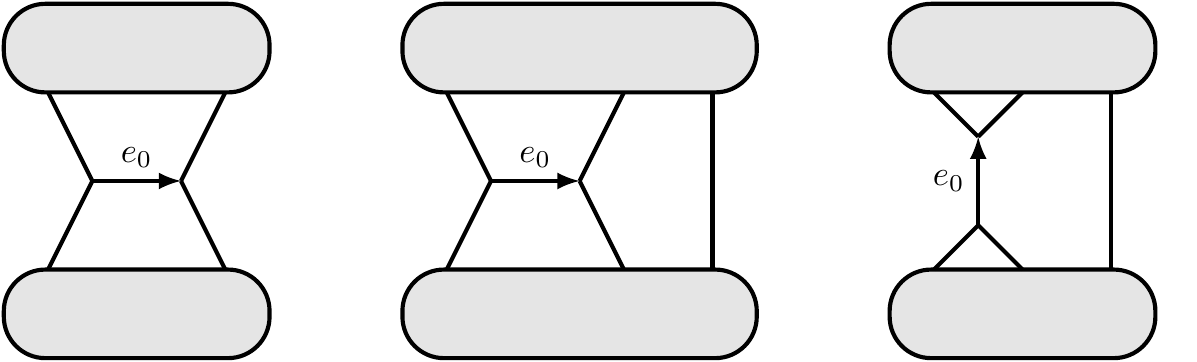}
  \caption{Exceptional cases A, B, C}
  \label{fig:3}
\end{figure}

\begin{enumerate}
\item Some $c_i \in \{0,4\}$
  \begin{enumerate}
  \item $c_1=0$ or $c_2=0$. This cannot happen in a bridgeless graph.
  \item $c_3=4$. This occurs only in graphs of type A.
  \item $c_3=0$. This is symmetric to $c_3=4$
  \item $c_1=4$. This happens exactly when $c_3=0$.
  \item $c_2=4$. This happens exactly when $c_3=4$.
  \end{enumerate}
\item Some $G_i$ contains a bridge.
  \begin{enumerate}
  \item $G_1$ contains a bridge. This is the situation of type B.
  \item $G_2$ contains a bridge. Symmetric to $G_1$ containing a
    bridge.
  \item $G_3$ contains a bridge. This is type C.
  \end{enumerate}
\item Some $c_i \in \{1,3\}$. That is, one of $e_1^* \pm e_2^*$ is a
  coedge. 
  \begin{enumerate}
  \item $c_2=1$. This is type C.
  \item $c_2=3$. This is type B.
  \item $c_1 \in \{1,3\}$. This is symmetric to $c_2 \in \{1,3\}$.
  \item $c_3=3$. This is type B.
  \item $c_3=1$. This is symmetric to $c_3=3$.
  \end{enumerate}
\end{enumerate}
If we choose $e_0$ to not lie in any 2-element cutset, then it becomes unnecessary to deal with case C. 
This is always possible, since by Lemma~\ref{cycle_lemma} every edge is part of a two element cutset
only if $G$ is cyclic. We proceed to construct strong emms for the remaining two cases:
\medskip

(A) We observe that $\ker(\phi_1)=\{0\}$. Moreover, since every (0,1) cycle in
$G$ is of the form $\phi_1(c)$ for some (0,1) cycle $c$ in $G_1$, by the
Lemma~\ref{coedge_lemma} every noncoedge in $H^1(G, \bZ))$ maps through
$\phi_1$ to a noncoedge in $H^1(G_1, \bZ)$.  Hence $\psi_1(q_1)$ satisfies the
conditions for a strong emm.  \medskip

(B) We have $c_2=c_3=3$; we construct the form $\frac 1 3 \psi_2(q_2)+
\frac 2 3 \psi_3(q_3)$. We need to show that this is greater than 1 on all
integral noncoedges. Equivalently (by Lemma~\ref{coedge_lemma}), we need
to show that for any $z \in H^1(G,\bZ)$ with $z(c) > 1 $ for some $(0,1)$-cycle
$c$ there is an $i \in \{2,3\}$ and $(0,1)$-cycle $c'$ in $G_i$ such that
$\phi_i(z)(c') > 1$.

We may assume that $c$ is the sum of at most two simple cycles (if
$z(\sum c_i)>1$ then $z(c_1)>1$ or $z(c_1+c_2)>1$ for some $c_1,c_2$).
Since $c$ is a (0,1) cycle it must be the image of a cycle in at least one of
$G_1$, $G_2$, $G_3$. If it is the image of a cycle in $G_2$ or $G_3$ there is
nothing to show, so we can assume that it is the image of some cycle $k$ in
$G_1$, but not of any cycle in $G_2$ or $G_3$. In order for this to be true $c$
must contain the edge $e_0$ and not the edge that becomes a bridge in $G_1$. 

By symmetry assume that $k$
contains the edge $e_1$ and not $e_2$. Note that since $c$ was a sum
of 2 (possibly trivial) simple cycles, so is $k$, say
$k=k_1+k_2$ where $k_1$ the simple summand containing $e_1$. Note
that, since $k_2$ contains neither $e_1$ nor $e_2$, the image of $k_2$
in $G$ is also the image of simple cycles in $G_2$ and $G_3$, so if
$\phi(z)(k_2) > 1$ the problem would be solved immediately. Hence we
will further assume $\phi(z)(k_1) \neq 0$.

We will proceed to use $k$ to construct a $(0,1)$ cycle $k'$ in $G_1$ that
contains both $e_1$ and $e_2$ and has $\phi_1(z)(k')>1$. If we can succeed in
doing this the result will follow, because depending on the relative
orientations of $e_1$ and $e_2$ in $k'$ there must be either a $(0,1)$ cycle
$c'$ in $G_2$ that maps to the same cycle in $G$ as $k'$ does (so
$\phi_2(z)(c') = \phi_1(z)(k') > 1$) or one in $G_3$ that maps to the same
cycle as $k'$ (so $\phi_3(z)(c') > 1$).

To build $k'$, find a simple cycle $l$ passing through $e_2$ that
intersects $k_2$ in a (possibly empty) arc (continuous path of edges). Such a
cycle always exists, since given any cycle containing $e_2$ (these exist by
connectedness) there is a maximal arc (possibly the whole cycle) disjoint from
$k_2$, and a (possibly empty) arc in $k_2$ joining its endpoints. We can choose
the arc in $k_2$ such that $l$+$k_2$ is also a $(0,1)$ cycle. But now since
$\phi_1(z)(k)>1$, out of the four $(0,1)$ cycles $k_1 \pm l$, $k_1 \pm (l+k_2)$
one of $\phi_1(z)(k_1 \pm l)$, $\phi_1(z)(k_1 \pm (l+k_2))$ must be greater
than 1. (Indeed, recalling our assumption that $\phi_1(z)(k_1) \neq 0$, if
$\phi_1(z)(l) \neq 0$ one of $k_1 \pm l$ will work, if $\phi_1(z)(k_2) \neq 0$
one of $k_1 \pm (l+k_2)$ will work, otherwise all four work .)  Calling this
cycle $k'$ the result follows.  \medskip

Hence the nongeneric cases are resolved. This concludes the proof of
Theorem~\ref{thm:Remm-cubic}.

\end{proof}

\begin{remark}
  In the induction argument of Theorem~\ref{thm:Remm-cubic}, it may
  happen that some of the incoming and outgoing edges are in fact the
  same. Then one of the graphs $G_i$ may have a component which is a
  ``loop with one edge $e$ and zero vertices''. We deal with this
  case formally, by taking $e^2$ to be the corresponding quadratic
  form.
\end{remark}

\begin{theorem}\label{thm:Remm-existence}
  Any graph $G$ admits a strong $\bQ$-emm.
\end{theorem}
\begin{proof}
  Existence of an $\bQ$-emm follows at once from
  Theorem~\ref{thm:Remm-cubic} and Lemma~\ref{lem:reduction} which
  reduces graph $G$ to a disjoint collection $G'=\sqcup G'_k$ of cubic
  graphs. However, the inclusion $S(G)\subset S(G')$ (see
  Definition~\ref{def:set-squares}) may be strict, so a 
  strong $\bQ$-emm $q$ of $G'$ may not be a strong $\bQ$-emm of $G$.

  By the result \cite[4.1]{ErdahlRyshkov_OnLatticeDicing}
  which we mentioned in Remark~\ref{rem:toroidal-Schottky}, the
  distinct vectors ${e_i^*}^2$ for the graph $G'$ are linearly
  independent. Thus, there exists a $q_0\in M_{\bQ}$ such that $q_0$
  is zero on $S(G)$ and positive on $S(G')\setminus S(G)$.  Then
  $q+\epsilon q_0$ for $0<\epsilon\ll1$ is a required strong
  $\bQ$-emm for $G$.
\end{proof}

\section{Concluding remarks and generalizations}
\label{sec:misc}

\subsection{Characterization of $\bZ$-emms of type $E_n$} It would be
interesting to find a geometric characterization of graphs admitting
$\bZ$-emms of types $E_6$, $E_7$, $E_8$, similar to the
characterization for $A_n$ and $D_n$ given in Theorem~\ref{thm:AnDn}.

\subsection{Special quadratic forms, and physical interpretation}
For any collection of positive real numbers
$(\lambda_1,\dotsc,\lambda_m)$ there is a natural positive definite
quadratic form $Q=\sum \lambda_i {e_i^*}^2$ on the \emph{homology
  group} $H_1(G,\bR)$. Since $Q$ is nondegenerate, we can use it to
identify $H_1$ and $H^1$, thus producing a positive definite quadratic
form $q$ on $H^1$. In coordinates, the matrix of $q$ is the inverse of
the matrix of $Q$.
Searching for an $\bR$-emm of this form leads to a system of $m$
nonlinear equations in $m$ variables which seems to be hard to
solve. We note that our solution for an $\bR$-emm is not of this
special form.

One can make a graph into an electric network by putting 
resistors $\lambda_i$ along the edges $e_i$. The \emph{total energy
  dissipation} of this electric system is $Q$. The condition
$q(e_i^*)=1$ can be reformulated in these terms as follows.
For any edge $e_i$, let $G_i$ be the graph obtained by cutting the
edge $e_i$ in the middle, thus producing two end points $p_i,
q_i$. Then the condition is that the resistance of $G_i$ between the
points $p_i$ and $q_i$ is 1, for each $i=1,\dotsc,m$.

\subsection{All dicing 2nd Voronoi cones}
The method of the proof of Theorem~\ref{thm:Remm-cubic} may also apply
to arbitrary, not necessarily cographic, regular matroids. This would
give a bigger open set in which $\oA_g\perf$ and $\oA_g\vor$ coincide.

\subsection{Other Torelli maps}
The proofs of Theorems~\ref{thm:crit-2nd-vor}, \ref{thm:crit-perf},
\ref{thm:crit-cent} work in a more general situation, if we replace
$S(G)$ by any finite set $\{ {v^*}^2 \}$ of symmetric rank~1 tensors.
Thus, they give regularity criteria for any rational map
$\oM\dashrightarrow\oA_g\utau$, $\tau=\tau\vor$, $\tau\perf$,
$\tau\cent$, for as long as $\oM$ is toroidal and the monodromy map has
a specific form $r_i\mapsto {v_i^*}^2$. For example, once properly
set up, this may apply to intermediate jacobians of cubic 3-folds.

\bibliographystyle{amsalpha}
%\bibliography{primary}
\newcommand{\etalchar}[1]{$^{#1}$}
\def\cprime{$'$}
\providecommand{\bysame}{\leavevmode\hbox to3em{\hrulefill}\thinspace}
\providecommand{\MR}{\relax\ifhmode\unskip\space\fi MR }
\providecommand{\MRhref}[2]{%
  \href{http://www.ams.org/mathscinet-getitem?mr=#1}{#2}
}
\providecommand{\href}[2]{#2}

\end{document}